\documentclass[12pt]{article}
\usepackage{amsmath, amsfonts, amsthm, amssymb, epsfig}

\newtheorem{theorem}{Theorem}
\newtheorem{lemma}{Lemma}

\begin{document}
\title{{\bf Pythagorean triangles within Pythagorean triangles }}
\author{Konstantine Zelator\\
Department of Mathematics\\
301 Thackeray Hall\\
139 University Place\\
The University of Pittsburgh\\
Pittsburgh, PA  15260\\
USA\\
and\\
P.O. Box 4280\\
Pittsburgh, PA  15203\\
kzet159@pitt.edu\\
e-mails: konstantine\underline{\ }zelator@yahoo.com}

\maketitle

\section{Introduction}

\vspace{.25in}

\parbox[b]{1.5in}{Suppose that $CBA$ is a Pythagorean triangle with sidelengths
  $\left| \overline{AB}\right| = c,\ \left|\overline{CA}\right|=b$, and
  $\left|\overline{CB}\right| =a$;  that is, a right triangle with the right
  angle at $C$; and with $a,b,c$ being positive integers such that
  $a^2+b^2=c^2$.  Then (without loss of generality -- $a$ and $b$ may be
  switched),}\hspace{.25in} \hspace{1.0in} \epsfig{file=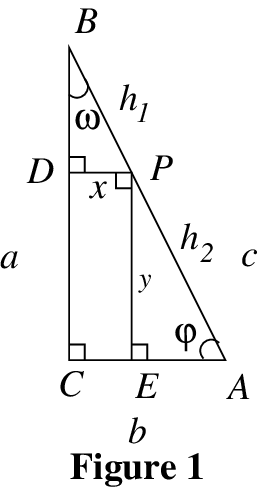,width=1.5in}

\begin{equation}\left\{ \begin{array}{l} a=d(m^2-n^2),\ b=d(2mn),\ c=d(m^2+n^2)\\
\\
{\rm where}\ d,m,n\ {\rm are\ positive\ integers\ such\ that}\\
\\
m > n,\ (m,n) =1,\ {\rm and}\ m+1\equiv 1({\rm mod}\ 2)\end{array} \right\}\label{E1}
\end{equation}

{\bf Note:}  Throughout this paper, $(X,Y)$ will stand for the greatest common
divisor of two integers $X$ and $Y$.

Thus, the condition $(m,n) =1$ says that $m$ and $n$ are relatively prime,
their greatest common divisor is $1$.  Also, the condition $m+n\equiv 1({\rm
  mod}2)$ says that $m$ and $n$ have different parities;  one of them is even,
the other odd.  The formulas in (\ref{E1}), are the well known parametric
formulas describing the entire family of Pythagorean triangles or triples.

A derivation of the formulas can be found in references \cite{1} and
\cite{2}.  For a wealth of historic information on Pythagorean triangles see
\cite{2} or \cite{3}.

Now, consider a point $P$ on the hypotenuse $\overline{AB}$, and let $D$ and
$E$ be the intersection points of the two lines through $P$ and parallel to
$\overline{CA}$ and $\overline{CB}$;  with the sides $\overline{CB}$ and
$\overline{CA}$ respectively.  Two right triangles are formed; the triangles
$BDP$ and $APE$.  Let $x$ and $y$ denote the lengths of line segments
$\overline{DP}$ and $\overline{PE}$ respectively.  Also, let $h_1 =
\left|\overline{BP}\right|$ and $h_2 = \left|\overline{AP}\right|$.  Then,

\begin{equation}
\left\{
\begin{array}{l}
\left|\overline{DP}\right| = \left|\overline{CE}\right| = x\ {\rm and}\
\left|\overline{PE} \right| = \left|\overline{DC}\right| = y.\\
\\
{\rm Thus},\ \left|\overline{BD}\right| = \left|\overline{BC}\right| -
\left|\overline{DC}\right| = a-y;\\ 
\\
{\rm and}\ \left|\overline{AE}\right| = \left|\overline{AC}\right| -
\left|\overline{CE}\right| = b-x \end{array} \right\} \label{E2}
\end{equation}

Both right triangles $BDP$ and $APE$ are similar to the right triangle of
$CBA$.  We have the similarity ratios,

\vspace{.15in}

$\left\{ \begin{array}{ccccc}\dfrac{x}{b} & = & \dfrac{a-y}{a} & = &
  \dfrac{h_1}{c} \\ 
\\
\dfrac{y}{a} & = & \dfrac{b-x}{b} & = & \dfrac{h_2}{c} \end{array}\right\}$
  \hspace{2.75in}   \begin{tabular}{l} (3i)\\ \\ \\ (3ii) \end{tabular}

\vspace{.15in}

\setcounter{equation}{3}

Since $a,b,c$ are (positive) integers, it follows, by inspection, from (3i)
that if one of $x,y$, or $h_1$ is a rational number, then all three of them
must be rational numbers.  Hence, either all three $x,y,h_1$ are rationals, or
otherwise, all three of them must be irrational.  Likewise, it follows from
(3ii) that either all three $x,y,h_2$ are rational or all three are
irrational.  Combining these two observations, we infer that

\vspace{.15in}

\fbox{\parbox{4.5in}{{\it Either all four $x,y,h_1,h_2$ are rational numbers
      or, otherwise, all four of them are irrationals.}}}

\vspace{.15in}

In Section 2, we state three lemmas from number theory.  One of them (Euclid's
Lemma) is well known.  We offer proofs for the other two.  

In Section 3, we prove Theorems 1 and 2; Theorem 2 is a corollary of Theorem
1.  

In Section 4,we consider and analyze three special cases.  These are the cases
when the point $P$ is the midpoint $M$ of the hypotenuse $\overline{AB}$; when
$P$ is the point $I$ where the angle bisector of the $90^{\circ}$ angle at $C$
intersects the hypotenuse $\overline{AB}$, and when the point $P$ is the foot
$F$ of the perpendicular from $C$ to the hypotenuse $\overline{AB}$.  

Back to Section 3.  In Theorem 1 we prove that  the two right triangles $BDP$
and $PEA$ in Figure 1 are either both Pythagorean or neither of them is a
Pythagorean triangle (assuming, of course, that $BCA$ is a Pythagorean
triangle).  It then follows, and this is part of Theorem 2, that when the
triangle $BCA$ is a primitive Pythagorean triangle, neither of the triangles,
$BDP$ and $PEA$ are Pythagorean for any position of the point $P$ along the
hypotenuse $\overline{AB}$.

In Section 5 (Theorem 6), we postulate that given a Pythagorean triangle with
side lengths $a=d(m^2-n^2),\ b=d(2mn)$, and $c=d(m^2+n^2)$, where $d,m,n$ are
positive integers such that $d \geq 2$, $(m,n) =1$, $m > n$, and $m+n \equiv
1({\rm mod}\ 2)$. Then there are exactly $d-1$ positions of the point $P$, such
that triangles $BDP$ and $PEA$ are both Pythagorean.

In Section 6, we will examine the general question of when, in addition to the
two triangles $BDP$ and $APE$ being Pythagorean, the four congruent right
triangles (within the rectangle $CDPE$) $CDP,\ CEP,\ DCE$, and $EPD$ are also
Pythagorean.  We derive a family of non-primitive Pythagorean triangles
$CBA$ with that property.

\noindent{\bf Note:}  In addition to the notation $(k,\ell)$ denoting the
greatest common divisor of two integers, $k$ and $\ell$, the notation $t|v$,
will stand for ``The integer $t$ is a divisor of the integer $v$''.

\section{Three lemmas from number theory}

\begin{lemma} (Euclid's Lemma):  Suppose that $a,b,c$, are natural numbers
  such that $c|ab$ (i.e., $c$ is a divisor of the product $ab$).  If
  $(c,a)=1$, then $c|b$.

\end{lemma}

For a proof of this well-known result, the reader
  may refer to \cite{1} or \cite{2}.

\begin{lemma}  Let $m,n$ be positive integers such that $m > n,\ (m,n)=1$, and
  $m+n \equiv 1({\rm mod}2)$.  Then

\begin{enumerate}
\item[(i)]  $(m^2+n^2,2mn) =1$

\item[(ii)]  $(m^2+n^2, m^2-n^2) = 1$

\item[(iii)]  $(m^2-n^2, 2mn)=1$
\end{enumerate}
\end{lemma}

\begin{proof}

\ \ \newline

 \begin{enumerate} 
\item[(i)]  We show that $m^2+n^2$ and $2mn$
    have no prime divisors in common.  If, to the contrary, $p$ were a prime
    divisor of both $m^2+n^2$ and $2mn$, then $p$ would be odd,  since
    $m^2+n^2 \equiv 1({\rm mod}2)$, by virtue of the hypothesis
    $m+n\equiv1({\rm mod}2)$.  Thus, $p|2mn$ implies, since $(p,2) =1$, that
    $p|mn$ (by Lemma 1).  But $p$ is a prime, so $p|mn$ implies that $p$ must
    divide at least one of $m,n$.  If $p|m$, then from $p|m^2+n^2$, it follows
    that $p|n^2$, and so $p|n$.  Thus, $p|m$ and $p|n$ contradicting the
    hypothesis that $(m,n) -1$

\item[(ii)]  A similar argument left to the reader ($p$ must divide the sum of
    $m^2+n^2$ and $m^2-n^2$, and their difference.  Hence, $p|2n^2$ and $p|2m^2$, which since $p$ is odd, eventually implies $p|n$ and $p|m$, a contradiction).

\item[(iii)]  A similar argument as in (i). 
\end{enumerate}

\end{proof}

\section{A theorem and a corollary}

\begin{theorem}  Suppose that $ABC$ is a Pythagorean triangle with the
  right angle at $C$; and with the three sidelengths satisfying the formulas in
  (\ref{E1}), namely $a=d(m^2-n^2),\ b=d(2mn),\ c=d(m^2+n^2)$, where $d,m,n$
  are positive integers such that $m>n,\ (m,n)=1$, and $m+n\equiv 1({\rm
  mod}\ 2)$.

Let $P$ be a point on the hypotenuse $\overline{AB}$, distinct from $A$ and
$B$.  Furthermore, suppose that $D$ is the foot of the perpendicular from $P$
to the side $\overline{CB}$; and $E$ the foot of the perpendicular from $P$ to
the side $\overline{CA}$, as in Figure 1.  Then, either both right triangles
$BDP$ and $APE$ are Pythagorean or neither of them is.

Moreover, if they are both Pythagorean, then the sidelengths\linebreak
$\left|\overline{BD}\right| = a-y,\ \left|\overline{DP}\right| =x$, and
$\left|\overline{BP}\right| = h_1$ of the triangle $BDP$ satisfy the formulas, 

$$ a-y=\delta (m^2-n^2),\ x= \delta(2mn),\ h_1 = \delta(m^2+n^2).
$$

While the sidelengths $\left|\overline{PE}\right|=y,\
\left|\overline{EA}\right| = b-x,\ \left|\overline{PA}\right| = h_2$ of the
triangle $PEA$ satisfy the formulas

$$
y=(d-\delta)(m^2-n^2),\ b - x = (d-\delta)(2mn),\ h_2 = (d-\delta)(m^2+n^2)
$$

\noindent where $\delta$ is a positive integer such that $1 \leq \delta \leq d-1$
\end{theorem}

\begin{proof} Suppose that the triangle $BDP$ is Pythagorean.  We will prove
  that the tringle $APE$ must also be Pythagorean; then so must the triangle $BDP$ be.

Since the triangle $BDP$ is Pythagorean, its three sidelengths, $x,\ a-y$, and
$h_1$ (see Figure 1) must be natural numbers.  From (3i) 

\begin{equation}
\begin{array}{rlll} \Rightarrow & x & = & \dfrac{b\cdot h_1}{c}
  \underset{{\rm by}\ (1)}{=}  \dfrac{d(2mn)}{d(m^2+n^2)} \cdot h_1;\\
\\
& x & = & \dfrac{2mnh_1}{m^2+n^2} \end{array} \label{E4} \end{equation}

From the conditions $(m,n)=1$ and $m+n \equiv 1({\rm mod}\ 2)$, it follows by
Lemma 2(i) that 

\vspace{.15in}

\hspace*{1.75in} $(m^2+n^2, 2mn)=1$ \hfill (4i)

\vspace{.15in}

Since $x$ is a natural number, equation (\ref{E4}) says that the integer
$m^2+n^2$ must be a divisor of the product $2mnh_1$, which clearly implies,
by (4i) and Lemma 1, that $h_1$ must be divisible by $m^2+n^2$.

\vspace{.15in}

\hspace{1.75in} $h_1 = \delta \cdot (m^2+n^2) $ \hfill (4ii)

\vspace{.15in}

\noindent for some positive integer $\delta$;  and since $h_1$ is the length
hypotenuse $\overline{BP}$ (triangle $BDP$), and the point $P$ lies strictly
between $A$ and $B$, it is clear that 

$$
h_1 = \left|\overline{BP}\right| < c = \left|\overline{BA}\right| =
d(m^2+n^2),$$

\noindent which together with (4ii) clearly show that 

\vspace{.15in}

\hspace{1.75in} $\begin{array}{l} 1 \leq \delta < d;\ {\rm or\
    equivalently},\\
\\
1 \leq \delta \leq d-1 \end{array}$ \hfill (4iii)

\vspace{.15in}

Note that by (4iii), we must have $d \geq 2$.  Going back to (\ref{E4}) and
using (4ii) we get

\vspace{.15in}

\hspace{1.75in} $x = (2mn)\delta$  \hfill (4iv)

\vspace{.15in}

and so, by (4iv), (3i), (4ii), and (\ref{E1}), we further obtain

\vspace{.15in}

\hspace{.5in} $\begin{array}{rcl} a-y & = & \delta(m^2-n^2);\ y= a-\delta
  (m^2-n^2);\\ 
\\
y & = & d(m^2-n^2) - \delta(m^2-n^2) = (d-\delta)(m^2-n^2) \end{array}$ \hfill
  (4v)

\vspace{.15in}

By using (1), (3i), and (4v) we also get

$$
b-x=(d-\delta)(2mn)$$

\noindent and 

$$h_2 = (d-\delta)(m^2+n^2).
$$

The proof is complete.  \end{proof}

\begin{theorem}  Let $CBA$ be a Pythagorean triangle, with the 90 degree 
angle at $C$.  Also, let $\left|\overline{CB}\right| = a,\
\left|\overline{CA}\right| = b$ and $\left|\overline{BA}\right| = c$, be 
the three sidelengths so that $a = d(m^2-n^2),\ b=d(2mn), c=d(m^2+n^2)$ where
$m,n,d$ are positive integers such that $m > n,\ (m,n) =1$, and $m+n \equiv
1({\rm mod}\ 2)$.

Let $P$ be a point on the hypotenuse $\left|\overline{BA}\right|$ and strictly
between the end-points $B$ and $A$.

Let $D$ and $E$ be the feet of the perpendiculars from the point $P$ to the
sides $\overline{CB}$ and $\overline{CA}$ respectively.

\begin{enumerate}
\item[(i)]  If $d=1$. i.e., if the Pythagorean triangle $CBA$ is primitive,
  then neither of the right triangles $PDB$ and $PEA$ is Pythagorean.

\item[(ii)]  If $d=2$, and the point $P$ is coincident with the midpoint $M$
  of the hypotenuse $\overline{BA}$, then both triangles $PDB$ and $PEA$  are
  Pythagorean.  Otherwise, if $P \neq M$, neither of these two triangles is
  Pythagorean.

\item[(iii)] If $d=3$, and the point $P$ is such that
  $\dfrac{\left|\overline{PB}\right|}{\left|\overline{PA}\right|} = \dfrac{1}{3}$ or
  $\dfrac{2}{3}$, then both triangles $PDB$ and $PEA$ are Pythagorean.
  Otherwise, if
  $\dfrac{\left|\overline{PB}\right|}{\left|\overline{PA}\right|} \neq
  \dfrac{1}{3},\ \dfrac{2}{3}$, then neither of these triangles are
  Pythagorean.
\end{enumerate}
\end{theorem}

\begin{proof} 

\begin{enumerate}\item[(i)]  If $d=1$, then neither of the two
  right triangles, $BDP$ and $PEA$ can be Pythagorean since according to
  Theorem 1, the natural number $\delta$ must
  satisfy $1 \leq \delta \leq d-1$, which is impossible when $d=1$.

\item[(ii)]  Suppose that $d=2$.  

If the point $P$ coincides with the midpoint
  $M$ of the hypotenuse $\overline{BA}$, then each of the triangles $BDP$ and
  $PEA$ is half the size of the triangle $CBA$.  So, by inspection,

$$
\begin{array}{rcl}
\left|\overline{BD}\right| = \left|\overline{PE}\right| & = & \dfrac{a}{2} =
\dfrac{2(m^2-n^2)}{2} = m^2-n^2 \\
\\
\left|\overline{DP}\right| = \left|\overline{EA}\right| & = & \dfrac{b}{2} =
\dfrac{2(2mn)}{2} = 2mn\\
\\
\left|\overline{BP}\right| = \left|\overline{PA}\right| & = & \dfrac{c}{2} =
\dfrac{2(m^2+n^2)}{2} = m^2+n^2,
\end{array}
$$

\noindent which proves that both triangles $BDP$ and $PEA$ are (in fact
primitive) Pythagorean triangles.   Conversely, if both triangles are
Pythagorean, then by Theorem 1, it follows that $1 \leq \delta \leq
d-1=2-1=1,$ \linebreak $ 1 \leq \delta \leq 1,\ \delta = 1$ which establishes that each of
the triangles is half the size of triangle of $CBA$; which implies that $P$ is
the midpoint of $\overline{BA}$.

\item[(iii)] Assume that $d=3$.  

Suppose $\dfrac{\left|\overline{PB}\right|}{\left|\overline{PA}\right|} =
\dfrac{1}{3}$ or $\dfrac{2}{3}$. If
$\dfrac{\left|\overline{PB}\right|}{\left|\overline{PA}\right|} =
\dfrac{1}{3}$, then the triangle $BDP$ is $\dfrac{1}{3}$ the size of triangle
$CBA$ and the triangle $PEA$ is $\dfrac{2}{3}$ the size of $CBA$.  We have,  

$$
\begin{array}{rcll}  \left|\overline{BD}\right| & = & \dfrac{a}{3} =
  \dfrac{3(m^2-n^2)}{3} = m^2-n^2, & \left|\overline{DP}\right| = \dfrac{b}{3}
  = \dfrac{3(2mn)}{3} = 2mn,\\
\\
\left|\overline{PB}\right| & = & \dfrac{c}{3} = \dfrac{3(m^2 + n^2)}{3} =
  m^2+n^2 & \end{array}
$$

\noindent and $\left|\overline{PE}\right| = \dfrac{2a}{3} = 2(m^2-n^2),\
\left|\overline{EA}\right| = \dfrac{2b}{3} = 2(2mn),$ \linebreak $
 \left|\overline{PA}\right| =
\dfrac{2c}{3}=2(m^2+n^2)$.  It is clear that both triangles $BDE$ and $PEA$
are Pythagorean.  

The argument for the case
$\dfrac{\left|\overline{PB}\right|}{\left|\overline{PA}\right|} =
\dfrac{2}{3}$ is similar (we omit the details).  

Now, the converse.  Assume that both triangles, $BDP$ and $PEA$, are
Pythagorean.  Then by Theorem 1 we must have,

$$
1 \leq \delta \leq d-1 = 3-1 = 2;\ \ \delta = 1\ {\rm or}\ 2.
$$

Using the formulas for the sidelengths (of triangles $BDP$ and $PEA$), found
in Theorem 1 we easily see that $\dfrac{\left|\overline{PA}\right|}{\left|\overline{PB}\right|} = \dfrac{1}{3}$, if $\delta = 1$.  While $\dfrac{\left|\overline{PA}\right|}{\left|\overline{PB}\right|}  = \dfrac{2}{3}$, if $\delta = 2$.  The proof is complete. 
\end{enumerate}
\end{proof}

\section{ Three special cases}

\begin{enumerate}
\item[A.]  {\bf Case 1: When the point $P$ is the midpoint $M$ of the
  hypotenuse $\overline{BA}$}

By inspection, it is clear that all six right triangles $BDP,\ PEA,\ CDP$,
$EPD,\ DCE$, and $PEC$ are all congruent and each of them is half the size of
triangle $BCA$.  Clearly then, by (\ref{E1}), these six triangles will be
Pythagorean if and only if the integer $d$ in (\ref{E1}) is even.

\begin{theorem} Let $BCA$ be a Pythagorean triangle with the 90 degree angle
  at $C$ and $\left|\overline{CB}\right| = a = d(m^2-n^2),\
  \left|\overline{CA} \right| = b = d(2mn),$ \linebreak 
$ \left|\overline{BA}\right| = c = d(m^2+n^2)$, where $d,m,n$ are postive
  integers such that $m > n,\ (m,n)=1$ and $m+n\equiv 1({\rm mod}\ 2)$.

Let $M$ be the midpoint of the hypotenuse $\overline{BA}$ and $D, E$ the feet
of the perpendiculars from $M$ to the sides $\overline{CB}$
and $\overline{CA}$ respectively (so $D$ and $E$ are the midpoints of
$\overline{CB}$ and $\overline{CA}$).  Then, the six right angles, $BDM,\
MEA,\ CDM,\ EMD,\ DCE$, and $MEC$ are congruent and have sidelengths as
follows:

\begin{tabular}{lll}
length of horizontal side & $= \dfrac{d}{2}(2mn)= dmn$\\
\\
length of veritical side & $= \dfrac{d(m^2-n^2)}{2}$\\
\\
length of hypotenuse & $= \dfrac{d(m^2+n^2)}{2}$
\end{tabular}

If $d$ is an even natural number, then the above six triangles are
Pythagorean;  otherwise, if $d$ is odd, they are non-Pythagorean.
\end{theorem}

\item[B.]  {\bf Case 2:  When the point $P$ is the foot $I$ of the angle
  bisector of the 90$^{\circ}$ angle at $C$}

\parbox[b]{2.0in}{Using the notation of Theorem 1, we have
\smallskip
 $\left|\overline{BD}\right| = a-y,\left|\overline{DI}\right| = x,\\
\\
\left|\overline{BP}\right| = h_1,\ \left|\overline{EA}\right| = b-x,\\
\\
\left|\overline{EI}\right| = y,\ {\rm  and}\  \left|\overline{IA}\right| =
 h_2.$}\hspace{1.25in} \epsfig{file=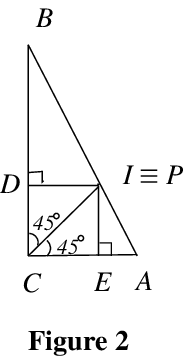,width=1.5in}

\vspace{.15in}

Clearly, we have $x=y$ in this case.  Note that the four
  congruent isosceles right triangles $DCI,\ IEC,\ DCE,\ DIE$ cannot be
  Pythagorean  (no Pythagorean triangle is isosceles).

By Theorem 1, the two right triangles $BDI$ and $IEA$ are either both
Pythagorean or neither of them are.  If they are both Pythagorean, then by
Theorem 1 we have, in particular, $x=\delta(2mn)$  and \linebreak
 $ y=(d-\delta)(m^2-n^2)$
with $m,n,d,\delta$ being positive integers such that $m>n,\ (m,n)=1,\
m+n\equiv1({\rm mod}\ 2)$ and $1 \leq \delta \leq d-1$ (and so $d \geq 2$).

Since $x=y$, we must have

\begin{equation}
\delta (2mn) = (d-\delta)(m^2-n^2) \label{E5}
\end{equation}

By Lemma 2(iii), we know that $(m^2-n^2, 2mn) = 1$.  So, by Lemma 1 and
(\ref{E5}) it follows that $2mn|d-\delta$ and $m^2-n^2 | \delta$ which, in
turn, leads to (when we go back to (\ref{E5}))

\vspace{.15in}

\hspace{1.0in} $\left\{ \begin{array}{l} \delta = K\cdot (m^2-n^2)\\
\\
 d-\delta = K \cdot (2mn),\\
\\
{\rm for\ some\ positive\ integer}\ K.\\
\\
{\rm Hence}\ d=k\cdot (m^2-n^2 + 2mn) \end{array} \right\}$ \hfill (5i)

\vspace{.15in}

\noindent Note that clearly, from (5i), $1 \leq \delta \leq d-1$. In fact, the
smallest possible value of $d$ is $7$; obtained for $K=1$ and $m=2,n=1$.
Moreover, $1\leq \delta \leq d-4$ since the smallest possible value of $K
\cdot (2mn)$ is $4$.

Using (5i) and Theorem 1, one can compute in terms of $m,n$, and $K$.  The
other four sidelengths of the triangles $BDI$ and $IEA$.  Also, by (5i) we get
$x=\delta(2mn) = K(2mn)(m^2-n^2) = y$.  We now state the following theorem.

\begin{theorem}  Let $CBA$ be a Pythagorean triangle with the $90^{\circ}$
  angle at $C$ and sidelengths given by

$$
\left|\overline{CB}\right| = a = d(m^2-n^2),\ \left|\overline{CA}\right| = b =
d(2mn),\ \left|\overline{BA}\right| = c= d(m^2+n^2),
$$

\noindent  where $d,m,n$ are positive integers such that, $m > n,\ m+n \equiv
1({\rm mod}\ 2)$, and $(m,n) = 1$.  Let $I$ be the foot of the perpendicular of
the angle bisector (of the $90^{\circ}$ angle at $C$) to the hypotenuse
$\overline{BA}$.

Also, let $D$ and $E$ be the feet of the perpendiculars from the point $I$ to
the sides $\overline{CB}$ and $\overline{CA}$ respectively.  Then, the two
right triangles $BDI$ and $IEA$ are both Pythagorean precisely when (i.e., if
and only if), $d=K \cdot (m^2-n^2 + 2mn)$ for some integer $K$. If $d= K\cdot
(m^2-n^2 + 2mn)$, then the sidelengths of triangle $BDI$ are given by  $\left|\overline{DI}\right| = x = K \cdot (2mn)(m^2-n^2) ,\ \ 
  h_1 = \left|\overline{BI}\right| = K(m^2-n^2)(m^2+n^2) = K(m^4-n^4)$,

and $\left| \overline{BD}\right| = a-y = K \cdot (m^2-n^2)^2$  and the
sidelengths of triangle $IEA$ are given by 

\vspace{.15in}

$\begin{array}{l} \left|\overline{IE}\right| = y =K(2mn)(m^2-n^2),\\
\\
\left|\overline{EA}\right| = b-x = K(2mn)(2mn) = K \cdot (2mn)^2,
\end{array}
$

and $h_2 = \left|\overline{IA}\right| = K \cdot (2mn)(m^2+n^2)$.

If the integer $d$ is not divisible by $m^2-n^2 + 2mn$, then neither of the
triangles, $BDI$ and $IEA$, is Pythagorean.
\end{theorem}

\item[C.] {\bf Case 3:  When the point $P$ is the foot $F$ of the
  perpendicular from the vertex $C$ to the hypotenuse $\overline{BA}$}

\vspace{.15in}

\parbox[b]{2.0in}{In this part, instead of using Theorem 1, we will first
  compute the sidelengths of the triangles $BDF,\ FEA$, and the four congruent
  triangles $FDC$, $DFE,\ DCE$,  and $CFE$ in terms of (the sidelengths) $a,b,c$.  After that we will implement the formulas in (\ref{E1}) in order to express the above sidelengths in terms of the integers $d,m,n$.} \hspace{1.0in} \epsfig{file=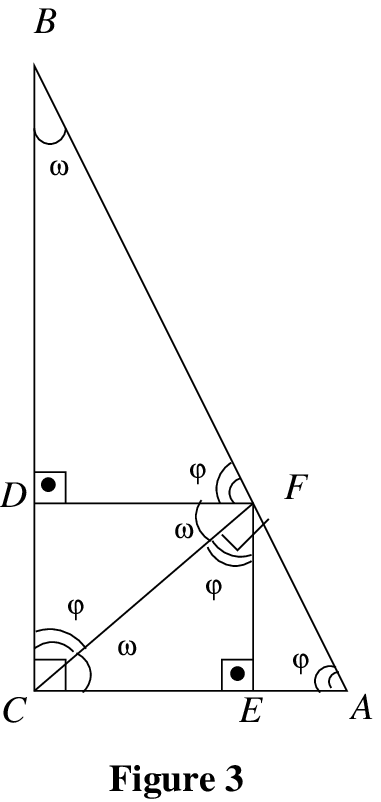,width=1.5in}

After that we will implement Lemma 2 to be able to draw the conclusions which
will lead to Theorem 5.  Note that since $F$ is the foot of the perpendicular
from $C$ to the hypotenuse $\overline{BA}$, the aforementioned six right
triangles are all similr to the triangle $CBA$.  Let $\omega$ and $\varphi$ be
the degree measures of the angles $\angle CBA$ and $\angle CAB$ respectively
(see Figure 3).

We have (and we set)

\begin{equation}
\left\{ \begin{array}{l}
\left|\overline{CB}\right| = a,\ \left|\overline{CA}\right| = b,\
\left|\overline{BA}\right| = c\\
\\
\left|\overline{DF}\right| = \left|\overline{CE} \right| = x,\
\left|\overline{EA} \right| = b-x \\
\\
\left|\overline{DC}\right| = \left|\overline{FE}\right| = y,\
\left|\overline{BD}\right| = a-y\\
\\
\left|\overline{BF} \right| = h_1,\  \left|\overline{FA} \right|=h_2,\ \left|
\overline{CF}\right| = \left|\overline{DE} \right|=h\end{array} \right\} \label{E6}
\end{equation}

Furthermore, 

$$
\sin \omega = \dfrac{y}{h} = \cos\varphi = \dfrac{b}{c} \ {\rm and}\ \cos
\omega = \dfrac{h}{b} = \dfrac{a}{c};
$$

\noindent and thus $h= \dfrac{ab}{c}$, which implies $y = h \cdot \cos \varphi
= h \cdot \dfrac{b}{c} = \dfrac{ab}{c} \cdot \dfrac{b}{c} =
\dfrac{ab^2}{c^2}$.  So, $a-y = a-\dfrac{ab^2}{c^2} = \dfrac{a(c^2-b^2)}{c^2}
= ({\rm since}\ c^2=a^2+b^2) \dfrac{a\cdot a^2}{c^2};\ a-y = \dfrac{a^3}{c^2}$

Next we calculate the lengths $x$ and $b-x$.  We have $\tan \omega = \cot
\varphi = \dfrac{y}{x}$; $\cot \omega = \tan \varphi = \dfrac{x}{y}$, and
$\tan \varphi = \dfrac{a}{b}$ which gives $\dfrac{x}{y} = \dfrac{a}{b}, x =
\dfrac{a}{b} \cdot y$.  Since $y = \dfrac{ab^2}{c^2}$ (see above), we obtain
$x = \dfrac{a}{b} \cdot \dfrac{ab^2}{c^2} = \dfrac{ba^2}{c^2}$.  From this we
get $b-x = b- \dfrac{ba^2}{c^2} = \dfrac{b(c^2-a^2)}{c^2} = \dfrac{b\cdot b^2}{c^2}= \dfrac{b^3}{c^2}$, since $c^2-a^2=b^2$.

Also, $\sin \omega = \dfrac{x}{h_1}; \ h_1 = \dfrac{1}{\sin \omega} \cdot x =
\dfrac{c}{b} \cdot \dfrac{ba^2}{c^2} = \dfrac{a^2}{c}$.  Similarly, we have
$\sin \varphi = \dfrac{y}{h_2};\ h_2 = \dfrac{1}{\sin \varphi} \cdot y =
\dfrac{c}{a} \cdot \dfrac{ab^2}{c^2} = \dfrac{b^2}{c}$.  We summarize these
lengths as follows:

\vspace{.15in}

\noindent {\it Sidelengths of triangle $BDF$}

\vspace{.15in}

\noindent $\left( \left|\overline{BD}\right| = a-y = \dfrac{a^3}{c^2},\ \left|
\overline{DF}\right| = x = \dfrac{ba^2}{c^2},\ \left| \overline{BF}\right| =
h_1 = \dfrac{a^2}{c} \right)$ \hfill (6i)

\newpage

\noindent{\it Sidelengths of triangle $FEA$}

\vspace{.15in}

\noindent$\left( \left|\overline{FE}\right| = y = \dfrac{ab^2}{c^2},\ \left|
\overline{EA}\right| = b-x = \dfrac{b^3}{c^2},\ \left| \overline{FA}\right| =
h_2 = \dfrac{b^2}{c} \right)$ \hfill (6ii)

\vspace{.15in}

\noindent {\it Sidelengths of the four congruent triangles $FDC,\ DFE,\ DCE,\
  CFE$}

\noindent $\begin{array}{rcl}\left( \left|\overline{DC}\right|\right. & = &
  \left|\overline{FE}\right| = y = \dfrac{ab^2}{c^2},\ \left|\overline{DF}\right| = \left|\overline{CE}\right|\\
\\
&  =& \left. x = \dfrac{ba^2}{c^2},\ \left|\overline{CF}\right| =
\left|\overline{DE}\right| = \dfrac{ab}{c} = h\right)\end{array}$ \hfill (6iii)

\vspace{.15in}

Next, we combine the length formulas in (6i), (6ii), and (6iii) with the
formulas in (\ref{E1}), since $CBA$ is a Pythagorean triangle, to obtain the
following.

\setcounter{equation}{6}

\begin{equation} \left\{ \begin{array}{rcl} a-y & = & \dfrac{d\cdot
      (m^2-n^2)^3}{(m^2+n^2)^2},\ x = \dfrac{d\cdot (m^2-n^2)^2
      \cdot(2mn)}{(m^2+n^2)^2} \\ 
\\
y & = & \dfrac{d\cdot (m^2-n^2) \cdot (2mn)^2}{(m^2+n^2)^2},\ b-x= \dfrac{d
      \cdot (2mn)^3}{(m^2+n^2)^2} \\
\\
h_1 & = & \dfrac{d\cdot (m^2-n^2)^2}{m^2+n^2},\ h_2 = \dfrac{d\cdot
      (2mn)^2}{m^2+n^2}\\
\\
h & =&  \dfrac{d\cdot (2mn)\cdot (m^2-n^2)}{m^2+n^2} \end{array}\right\}\label{E7}
\end{equation}

The following lemma from number theory is well-known and comes in handy.

\begin{lemma} Let $i_1,\ i_2,\ i_3,\ e_1,\ e_2, \ e_3$ be positive integers
  such that  $(i_1,i_2) = 1 = (i_1, i_3)$.  Then,

\begin{enumerate}
\item[(a)]  $\left( i^{e_1}_1,\ i^{e_2}_2\right) = 1$

\item[(b)]  $\left( i^{e_1}_1,\ i^{e_2}_2 \cdot i^{e_3}_3 \right) = 1$
\end{enumerate}

It follows from Lemmas 2 and 3 that 

\hspace*{-.5in}\begin{equation}
\left\{ \begin{array}{l}
\left(\left( m^2+n^2\right)^2, \ \left(m^2-n^2 \right)^3\right)=1,\\
\\ 
\left(\left( m^2+n^2 \right)^2, \ \left(2mn\right)^2\right) = 1\\
\\
\left( m^2+n^2,\ \left(m^2- n^2\right)^2\right) = 1,\\
\\
 \left(m^2+n^2,\ \left(2mn\right)^2 \right) = 1\\
\\
\left(m^2+n^2,\ \left(2mn\right) \cdot \left(m^2-n^2\right)\right) = 1 \\
\\
\left(\left(m^2+n^2\right)^2,\ \left(m^2-n^2\right)^2 \cdot \left(2mn\right)
\right) = 1\\
\\
\left(\left(m^2+n^2\right)^2,\ \left(m^2-n^2\right) \cdot \left(2mn\right)^2
\right) = 1 
\end{array} \right\} \label{E8} \end{equation}
\end{lemma}

A careful look at formulas (\ref{E7}) and the coprimeness conditions in
(\ref{E8}), in conjunction with Lemma 1, reveals that either all six
triangles, $BDF,\ FEA,\ FDC,\ DFE,\ DCE$, and $CFE$ are Pythagorean; or none
of them are.

They are all Pythagorean precisely (i.e., if and only if) the integer $d$ is
divisible by $(m^2+n^2)^2$, i.e., when 

\begin{equation}
\left\{  \begin{array}{l} d= K \cdot (m^2+n^2)^2 \\
\\
{\rm for\ some\ positive\ integer}\ K\end{array}\right\} \label{E9} 
\end{equation}

This is precisely when all seven numbers $y,\ a-y,\ x,\ b-x,\ h_1,\ h_2$, and
$h$ are integers.  When (\ref{E9}) holds true, we can compute, via (\ref{E7})
  and (61), (6ii), and (6iii) all the sidelengths in terms of the integers
  $m,n$, and $K$.

We have the following theorem.

\begin{theorem}  Let $CBA$ be a Pythagorean triangle, with the $90$-degree
  angle at $C$.  With $\left|\overline{CB}\right| = a= d(m^2-n^2),\
  \left|\overline{CA}\right| = b = d(2mn),\ \left|\overline{BA}\right| =
  d(m^2+n^2) = c$, where $d,m,n$ are positive integers such that $m > n,\
  (m,n)=1$, and $m+n \equiv 1({\rm mod} 2)$.  

Also, let $F$ be the foot of the perpendicular from the vertex $C$ to the
hypotenuse $\overline{BA}$.  Then, the six similar triangles $BDF,\ FEA$, (and
the four congruent ones) $FDC,\ DFE,\ DCE,\ CFE$ are either all Pythagorean or
none of them are.  They are all Pythagorean precisely when (i.e., if and only
if)  $d=K\cdot (m^2+n^2)^2$, for some positive integer $K$.  When $d$
satisfies the said condition, the sidelengths of the above six triangles are
given by the following formulas.  

\vspace{.15in}

\noindent For triangle $BDF$

\vspace{.15in}

\noindent $\left| \overline{BD}\right| = a-y = K \cdot \left(m^2-n^2\right),\
\left|\overline{DF} \right| = x = K \cdot \left( m^2 - n^2\right)^2 \cdot
(2mn), \ {\rm and}\ h_1 = K \cdot \left(m^2+n^2\right) \cdot
\left(m^2-n^2\right)^2$

\vspace{.15in}

\noindent For triangle $FEA$:

\vspace{.15in}

$\left| \overline{FE}\right| = y = K \cdot \left(m^2-n^2\right) \cdot (2mn)^2,\
\left|\overline{EA} \right| = b -x = K\cdot (2mn)^3$,  and $h_2 = K \cdot
\left(m^2+n^2\right) \cdot (2mn)^2$.

\vspace{.15in}

\noindent For the four congruent triangles $FDC,\ DFE,\ DCE, CFE$:

\vspace{.15in}

$\left|\overline{DC} \right| = \left|\overline{FE}\right| = y = K \cdot
\left(m^2-n^2\right) \cdot \left(2mn\right)^2$,

\vspace{.15in}

$\left|\overline{DF} \right| = \left| \overline{CE} \right| = x = K \cdot
\left( 2mn\right) \cdot \left(m^2 - n^2\right)^2,$

\noindent  and 

\vspace{.15in}

\noindent $h= \left|\overline{CF}\right| = \left| \overline{DE}\right| =
K\cdot (2mn) \cdot \left(m^2-n^2\right) \cdot \left(m^2 + n^2\right) = K \cdot
(2mn)\left(m^4-n^4\right)$.
\end{theorem}
\end{enumerate}

\noindent {\bf Numerical Examples}

If we take $K=1$ and $mn \leq 4$, then $K = 1$ and $m=2,\ n=1$; or $K=1$ and $m =
4, n=1$.

\begin{enumerate}
\item[(a)]  $K=1,\ m=2,\ n=1$.  We obtain the following:  

\vspace{.15in}

$\begin{array}{l} d=1 \cdot \left(2^2+1^2\right)^2 = 5^2 = 25, \ h=60,\ h_1 = 45,\
  h_2 = 80,\\
\\
y= 48,\ a-y = 75-48 = 27,\ x = 36,\ b-x = 100 - 36 = 64,\\
\\
a = 75,\ b = 100,\ c = 125
\end{array}
$

\item[(b)]  $K = 1, \ m=4,\ n=1$.  We have the following:

\vspace{.15in}

$\begin{array}{l}
d=289,\ a-y = 15,\ x=1800,\ h_1 = 3825,\\
\\
y = 960,\ b-x = 512,\ h_2 = 1088,\ h = 1404\\
\\
a = 4335,\ b= 2312,\ c = 4913
\end{array}
$
\end{enumerate}

\section{Exactly $(d-1)$ positions of $P$}

Given a Pythagorean triangle $CBA$, as in Figure 1, and with the point $P$ on
the hypotenuse $\overline{BA}$, and $D$ and $E$ being the perpendicular
projections of $P$ on the sides $\overline{CB}$ and $\overline{CA}$
respectively.  We know from Theorem 1 that either both triangles $BDP$ and
$PEA$ are Pythagorean, or neither of them are.  The integer $\delta$, as
described in Theorem 1 must satisfy $1 \leq \delta \leq d-1$; which means that
$d\geq 2$ is a necessary condition.  There are $(d-1)$ choices for $\delta$.
If we subdivide the hypotenuse $\overline{BA}$ into $d$ equal length segments,
each segment having length $m^2+n^2$, it is easily seen that for each such
position of the point $P$ both triangles $BDP$ and $PEA$ are Pythagorean.
There are exactly $(d-1)$ such positions for the
point $P$ along the hypotenuse $\overline{BA}$.  These are the points $P_1,
\ldots , P_{d-1}$; so that each of the consecutive line segments
$\overline{BP}_1,\overline{P_1P_2} , \ldots , \overline{P_{d-1}A}$ (exactly
$d$ line segments)  has length $m^2+n^2$.

We postulate the following theorem.

\begin{theorem}
Let $CBA$ be a Pythagorean triangle with the $90^{\circ}$ angle at the vertex
$C$.  With sidelengths given by $\left| \overline{CB}\right| = a =
d(m^2-n^2)$, \linebreak $\left|\overline{CA}\right| = b = d(2mn),\
\left|\overline{BA}\right| = d(m^2+n^2)$, where $d,m,n$ are positive integers
such that $d\geq 2,\ m > n,\ (m,n)=1$, and $m+n \equiv ({\rm mod}\ 2)$.  Also,
let $P_1,\ldots, P_{d-1}$ be the $(d-1)$ points on the hypotenuse
$\overline{BA}$ such that the $d$ consecutive line segments $\overline{BP}_1,\
\overline{P_1P_2}, \ldots , \overline{P_{d-1}A}$ have equal lengths; each
having length $m^2 + n^2$.  Then there are exactly $(d-1)$ points $P$ on the
hypotenuse $\overline{BA}$ such that both trangles $BDP$ and $PEA$ are
Pythagorean where $D$ and $E$ are the feet of the perpendiculars from $P$ to
the sides $\overline{CB}$ and $\overline{CA}$ respectively.  These $(d-1)$
points are precisely the points $P_1, \ldots ,P_{d-1}$ described above.
Furthermore, each pair of Pythagorean triangles $BD_iP_i$  and $P_iE_iA$ have
sidelengths given by $\left|\overline{BD}_i\right| = i \cdot
\left(m^2-n^2\right),\ \left|\overline{D_iP_i}\right| = i(2mn),\
\left|\overline{BP_i}\right| = i\left(m^2+n^2\right),\ \left|
\overline{P_iE_i}\right| = (d-i)\left(m^2-n^2\right),\ \left|\overline{E_iA}
\right| = (d-i)(2mn),\ \left|\overline{P_iA}\right| =
(d-i)\left(m^2+n^2\right)$, for $i=1,\ldots , d-1$; and where $D_i$ and $E_i$
are the perpendicular projections of the point $P_i$ onto the sides
$\overline{CB}$ and $\overline{CA}$ respectively.
\end{theorem}

\section{Other cases}

In this section, we explore the following question.  If in addition to the two
triangles in Figure 1, $BDP$ and $PEA$ being Pythagorean, we require that the
four congruent triangles $DCE,\ PEC,\ CDP,\ EPD$, also be Pythagorean.  What
are the necessary and sufficient conditions for this to occur?  

For these four congruent triangles to be Pythagorean, the integers\linebreak  $x=
\left|\overline{DP}\right| = \left|\overline{CE}\right|$ and $y=
\left|\overline{DC}\right| = \left|\overline{PE}\right|$ must satisfy the
condition,

$$
x^2+y^2 = \ {\rm perfect\ square}.
$$

\noindent Combining this with Theorem 6 leads to the following theorem.

\begin{theorem}
Let $CBA$ be a Pythagorean triangle with the $90$ degree angle at the vertex
$C$; and with sidelengths, $a=\left|\overline{CB}\right| =
d\left(m^2-n^2\right)$, $b = \left|\overline{CA}\right| = d(2mn),\ c =
\left|\overline{BA} \right| = d\left(m^2+n^2\right)$ where $d,m,n$ are
positive integers such that $d\geq 2,\ m>n,\ (m,n)=1$, and $m+n \equiv 1({\rm
  mod}\ 2)$.  Let $P$ be a point on the hypotenuse $\overline{BA}$, and $D$ and
$E$ be the feet of the perpendiculars from the point $P$ onto the sides
$\overline{CB}$ and $\overline{CA}$ respectively.  Also, let  $x=\left|
\overline{DP}\right| = \left|\overline{CD}\right|,\ y =
\left|\overline{DC}\right| = \left|\overline{PE}\right|$ so that 

$$
a-y = \left|\overline{BD}\right|\ {\rm and}\ b-x =
\left|\overline{EA}\right|.$$

\noindent Then, the two right triangles $BDP$ and $PEA$, as well as the four
congruent triangles, $DCE,\ PEC,\ CDP,\ EPD$, are all (six triangles) are
Pythagorean if and only if there exist positive integers $D,M,N$ such that 

$$
M > N,\ (M,N)=1,\ M+N \equiv 1({\rm mod}\ 2)
$$

\noindent and with either

\vspace{.15in}

$\left\{ \begin{array}{l} y = \delta \left(m^2-n^2\right) = D \cdot \left(
  M^2-N^2\right)\\
\\
x = (d-\delta) \cdot (2mn) = D\cdot (2MN)\\
\\
\delta\ {\rm a\ positive\ integer\ such \ that\ } 1 \leq \delta \leq
  d-1\end{array}\right\}$ \hfill (10i)

\vspace{.15in}

\noindent or

\vspace{.15in}

$\left\{ \begin{array}{l} y= \delta \left(m^2-n^2\right) = D \cdot (2MN)\\
\\
x = (d-\delta)(2mn) = D\cdot \left(M^2-N^2\right) \\
\\
\delta\ {\rm a \ positive\ integer\ such \ that\ } 1 \leq \delta \leq
d-1\end{array}\right\}$ \hfill (10ii)
\end{theorem}

\vspace{.15in}

The following example shows that there exist nonprimitive Pythagorean
triangles such that there is no point $P$ on the hypotenuse $\overline{BA}$
such that all six triangles $BDP,\ PEA,\ DCE,\ PEC,\ CDP,\ EPD$, are
Pythagorean.

\vspace{.15in}

\noindent{\bf Example:}  Take $d=5,\ m=2,\ n=1$.  Then the sidelengths of
triangle $CBA$ are $a=5\cdot \left(2^2-1^2\right) = 15,\ b = 5 \cdot (2\cdot 2
\cdot 1)=20$, and $c= 5\cdot \left(2^2+1^1\right) = 25$.  The possible values
of the integer $\delta$ are $\delta = 1,2, \ldots , d-1 = 1,2,3,4$.  Using the
formulas $y=\delta \left(m^2-n^2\right)$ and $x = (d-\delta)(2mn)$ we have the
following.

\begin{enumerate}
\item[1.]  $\delta = 1:\ y=3,\ x=(5-1)\cdot 4 = 16$

\noindent and $y^2+x^2 = 9 + 256 = 265$ not an integer square.

\item[2.] $\delta = 2,\ y=2\cdot 3=6,\ x= (5-2) \cdot 4 = 12$

\noindent and $y^2+x^2 = 36 + 144 = 180$, not a perfect square.

\item[3.]  $\delta = 3,\ y= 3\cdot 3 = 9,\ x= (5-3)\cdot 4 = 8$

\noindent and $y^2+x^2 = 81 + 64 = 145$, not an integer square.\

\item[4.] $\delta = 4,\ y=4\cdot 3 = 12,\ x=(5-4)\cdot 4 =4$

\noindent and $y^2+x^2 = 144 + 16 = 160$, not a perfect square.
\end{enumerate}

There are many ways in which one can use the conditions (10i) or (10ii) of
Theorem 7 in order to produce families of Pythagorean triangles such that
each member  (of those families) has the property that there is a point $P$ on
its hypotenuse such that all six triangles (as described in Theorem 7) are
Pythagorean.  We produce such a family.

\vspace{.15in}

\noindent {\bf Family 1:}  Consider (10i):

\vspace{.15in}

\hspace{.5in} $\left.\begin{array}{rcl} y & = & \delta \left(m^2-n^2\right)
  =D\left(M^2-N^2 \right)\\
\\
x & = & (d-\delta)(2mn) = D(2MN) \end{array}\right\}$ \hfill (10i)

\vspace{.15in}

Let $K$ be a positive integer.

Take $D= K \cdot mn\left(m^2-n^2\right)$.

\vspace{.15in}

From the second equation in (10i) we obtain

$$
d-\delta = K\cdot MN\left(m^2-n^2\right) $$

\noindent and from the first equation in (10i) we get 

$$
\delta = Kmn\left(M^2 -N^2\right).
$$

\noindent Hence, $d = \delta + KMN\left(m^2-n^2\right) = K \cdot \left[
  mn\left(M^2 - N^2\right) + MN \left(m^2-n^2\right) \right]$.  Obviously $1
  \leq \delta \leq d - 1$ and $d\geq 2$, as required. We have the following.

\begin{center}

{\bf Family 1}

\end{center}

{\it Let $m,n,M,N$ be positive integrs such that $m > n,\ (m,n)=1$, \linebreak
  $m+n\ \equiv({\rm mod}\ 2),\ M>N,\ (M,N) = 1\ M+N \equiv 1({\rm mod}\ 2)$.  Also, let
  $K$ be a positive integer and $\delta = Kmn \left(M^2-N^2\right),\ d= K\cdot
  \left[ mn\left(M^2-N^2\right) + MN \left(m^2-n^2\right)\right]$.  Consider
  the Pythagorean triangle $CBA$ with sidelengths $\left|\overline{CB}\right|
  = a = d \left(m^2-n^2\right),$\linebreak $ \left| \overline{CA}\right| = b = d(2mn),\
  \left|\overline{BA}\right| = c = d\left(m^2 + n^2\right)$.  Let $P$ be the
  point on the hypotenuse $\left|\overline{BA}\right|$ such that
  $\left|\overline{BP}\right| = h_1 = \delta \left(m^2 + n^2 \right)$; and let
  $D$ and $E$ be the perpendicular projections of $P$ onto the sides
  $\overline{CB}$ and $\overline{CA}$ respectively.  Then all six right
  triangles $BDP,\ PEA,\ DCE,\ PEC,\ CDP$ and $EPD$ are Pythagorean.}

\end{document}